\documentclass{amsart}
\usepackage{amsfonts}
\usepackage{mathrsfs}
\usepackage{amsmath,amssymb}

\theoremstyle{plain}
\newtheorem{theorem}{Theorem}[section]
\newtheorem{lemma}{Lemma}[section]
\newtheorem{proposition}{Proposition}[section]
\newtheorem{corollary}{Corollary}[section]

\theoremstyle{definition}
\newtheorem{definition}{Definition}[section]
\newtheorem{example}{Example}[section]

\theoremstyle{remark}

\newtheorem{question}{Question}[section]

\begin{document}

\title[Some weak versions of the $M_{1}$-spaces]
{Some weak versions of the $M_{1}$-spaces}

\author{Fucai Lin}
\address{Fucai Lin(corresponding author): Department of Mathematics and Information Science,
Zhangzhou Normal University, Zhangzhou 363000, P. R. China}
\email{linfucai2008@yahoo.com.cn}

\author{Shou Lin}
\address{(Shou Lin)Institute of Mathematics, Ningde Teachers' College, Ningde, Fujian
352100, P. R. China} \email{linshou@public.ndptt.fj.cn}

\thanks{Supported by the NSFC (No. 10971185, No. 10971186) and the Educational Department of Fujian Province (No. JA09166) of China.}


\begin{abstract}
We mainly introduce some weak versions of the $M_{1}$-spaces, and study some properties about these spaces.
The mainly results are that: (1) If $X$ is a compact scattered space and $i(X)\leq 3$, then $X$ is an $s$-$m_{1}$-space; (2) If $X$ is a strongly monotonically normal space, then $X$ is an $s$-$m_{2}$-space; (3) If $X$ is a $\sigma$-$m_{3}$ space, then $t(X)\leq c(X)$, which extends a result of P.M. Gartside in \cite{CP}. Moreover, some questions are posed in the paper.
\end{abstract}

\maketitle

\section{Introduction}
All spaces are $T_1$ and regular unless stated otherwise, and all
maps are continuous and onto.
The letter $\mathbb{N}$ denotes the set of all positively natural numbers.
Let $X$ be a topological space. Recalled that a family $\mathcal{P}$ of subsets of $X$ is called {\it closure-preserving} if, for any $\mathcal{P}^{\prime}\subset \mathcal{P}$, we have $\overline{\bigcup\mathcal{P}^{\prime}}=\bigcup\{\overline{P}: P\in\mathcal{P}^{\prime}\}$. Moreover, $X$ is called
an $M_{1}$-space \cite{CJ} if $X$ has a $\sigma$-closure-preserving base. It is still a famous open problem (usually called the $M_{1}=M_{3}$ question, see \cite{Gr1984}) whether each stratifiable space is $M_{1}$. M. Ito proved that every $M_{3}$-space with a closure-preserving local base at each point is $M_{1}$ \cite{IM}, and T. Mizokami has just showed that every $M_{1}$-space
has a closure-preserving local base at each point \cite{MT}. Therefore, to give a positive answer to $M_{1}=M_{3}$, it is sufficient to prove that each stratifiable space has a closure-preserving local base at each point.

R.E. Buck first introduced and studied the $m_{i}$ (see Definition~\ref{d0}) properties in \cite{BR}, where he gave some interesting and surprising results about $m_{i}$-spaces. Recently, A. Dow, R. Mart$\acute{\imath}$nez and V.V. Tkachuk have also make some study on spaces with a closure-preserving local base at each point (In fact, they call such spaces for {\it Japanese spaces} in their paper.) \cite{DA}. In this paper we introduce some weak versions of $M_{1}$-spaces, and study some properties and relations on these spaces.

\maketitle

\section{Preliminaries}
\begin{definition}
Let $X$ be a topological space and $\mathcal{B}$ a family of subsets of $X$. $\mathcal{B}$ is called a {\it quasi-base} \cite{Gr1984}
for $X$ if, for each $x$ and open subset $U$ with $x\in U$, there exists a $B\in\mathcal{B}$ such that $x\in\mbox{int}(B)\subset B\subset U$.
\end{definition}

\begin{definition}
Let $X$ be a topological space and $\mathcal{P}$ a pair-family for
$X$, where, for any $P\in
\mathcal{P}$, we denote $P=(P^{\prime}, P^{\prime\prime})$. $\mathcal{P}$ is called a {\it pairbase} \cite{Gr1984} if $\mathcal{P}$
satisfies the following conditions:
\begin{enumerate}
\item For any $(P^{\prime},
P^{\prime\prime})\in \mathcal{P}$, $P^{\prime}\subset P^{\prime\prime}$ and $P^{\prime}$ is open subset of $X$;

\item For any $x\in U\in \tau (X)$, there exists $(P^{\prime},
P^{\prime\prime})\in \mathcal{P}$ such that $x\in P^{\prime}\subset
P^{\prime\prime}\subset U$.
\end{enumerate}

Moreover, a pairbase $\mathcal{P}$ is called a {\it cushioned} if , for each $\mathcal{P}^{\prime}\subset\mathcal{P}$, we have $\overline{\bigcup\{P^{\prime}: (P^{\prime}, P^{\prime\prime})\in\mathcal{P}^{\prime}\}}\subset\bigcup\{P^{\prime\prime}: (P^{\prime}, P^{\prime\prime})\in\mathcal{P}^{\prime}\}$
\end{definition}

\begin{definition}
Let $\mathcal{P}$ be a collection of subsets of $X$.
$\mathcal{P}$ is called {\it closure-preserving} \cite{Gr1984} if, for any $\mathcal{P}^{\prime}\subset \mathcal{P}$, we have $\overline{\bigcup\mathcal{P}^{\prime}}=\bigcup\{\overline{P}: P\in\mathcal{P}^{\prime}\}$.
\end{definition}

A family $\mathcal{A}$ of open subsets of a space $X$ is called {\it a base of $X$ at a set $A$} if $A=\bigcap\mathcal{A}$ and for any neighborhood $U$ of $A$, there is a $V\in\mathcal{A}$ such that $A\subset V\subset U$. A family $\mathcal{A}$ of subsets of a space $X$ is called a {\it quasi-base of $X$ at a set $A$} if $A=\bigcap\mathcal{A}$ and for any neighborhood $U$ of $A$, there is a $V\in\mathcal{A}$ such that $A\subset\mbox{int}(V)\subset V\subset U$.

\begin{definition}\label{d0}
Let $X$ be a space, $x\in X$ and $F$ a closed subset of $X$. Then
\begin{enumerate}
\item $X$ is $m_{1}$ ({\it $\sigma$-$m_{1}$}) {\it at the point $x$} \cite{BR} if $X$ has a closure-preserving ({\it $\sigma$-closure-preserving}) local base at the point $x$. $X$ is called an {\it $m_{1}$-space} ({\it $\sigma$-$m_{1}$-space}) if every point of $X$ is $m_{1}$ ($\sigma$-$m_{1}$);

\item $X$ is {\it s-$m_{1}$} ({\it s-$\sigma$-$m_{1}$}) {\it at $F$}  if $X$ has a closure-preserving ({\it $\sigma$-closure-preserving}) local base at $F$. $X$ is called an {\it s-$m_{1}$-space} ({\it s-$\sigma$-$m_{1}$-space}) if every closed subset of $X$ is $m_{1}$ ($\sigma$-$m_{1}$);

\item $X$ is $m_{2}$ ({\it $\sigma$-$m_{2}$}) {\it at the point $x$} \cite{BR} if $X$ has a closure-preserving ({\it $\sigma$-closure-preserving}) local quasi-base at the point $x$. $X$ is called an {\it $m_{2}$-space} ({\it $\sigma$-$m_{2}$-space}) if every point of $X$ is $m_{2}$ ($\sigma$-$m_{2}$);

\item $X$ is {\it s-$m_{2}$} ({\it s-$\sigma$-$m_{2}$}) {\it at $F$}  if $X$ has a closure-preserving ({\it $\sigma$-closure-preserving}) local quasi-base at $F$. $X$ is called an {\it s-$m_{2}$-space} ({\it s-$\sigma$-$m_{2}$-space}) if every closed subset of $X$ is $m_{2}$ ($\sigma$-$m_{2}$);

\item $X$ is $m_{3}$ ({\it $\sigma$-$m_{3}$}) {\it at the point $x$} \cite{BR} if $X$ has a cushioned local pairbase ({\it $\sigma$-cushioned local pairbase}) at the point $x$. $X$ is called an {\it $m_{3}$-space} ({\it $\sigma$-$m_{3}$-space}) if every point of $X$ is $m_{3}$ ($\sigma$-$m_{3}$);

\item $X$ is {\it s-$m_{3}$} ({\it s-$\sigma$-$m_{3}$}) {\it at $F$}  if $X$ has a cushioned local pairbase ({\it $\sigma$-cushioned local pairbase}) at $F$. $X$ is called an {\it s-$m_{3}$-space} ({\it s-$\sigma$-$m_{3}$-space}) if every closed subset of $X$ is $m_{3}$ ($\sigma$-$m_{3}$).

\end{enumerate}
\end{definition}

It is easy to see that

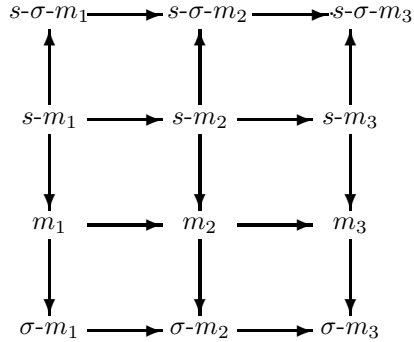
\setlength{\unitlength}{1cm}
\begin{picture}(15,5.5)\thicklines
 \put(3,0){\makebox(0,0){$\sigma$-$m_{1}$}}
 \put(3.5,0){\vector(1,0){1}}
 \put(5,0){\makebox(0,0){$\sigma$-$m_{2}$}}
 \put(5.5,0){\vector(1,0){1}}
 \put(7,0){\makebox(0,0){$\sigma$-$m_{3}$}}
 \put(3,1.2){\vector(0,-1){1}}
 \put(5,1.2){\vector(0,-1){1}}
 \put(7,1.2){\vector(0,-1){1}}
 \put(3,1.4){\makebox(0,0){$m_{1}$}}
 \put(5,1.4){\makebox(0,0){$m_{2}$}}
 \put(7,1.4){\makebox(0,0){$m_{3}$}}
 \put(3.5,1.4){\vector(1,0){1}}
 \put(5.5,1.4){\vector(1,0){1}}
 \put(3,2.6){\vector(0,-1){1}}
 \put(5,2.6){\vector(0,-1){1}}
 \put(7,2.6){\vector(0,-1){1}}
 \put(3,2.8){\makebox(0,0){$s$-$m_{1}$}}
 \put(5,2.8){\makebox(0,0){$s$-$m_{2}$}}
 \put(7,2.8){\makebox(0,0){$s$-$m_{3}$}}
 \put(3.5,2.8){\vector(1,0){1}}
 \put(5.5,2.8){\vector(1,0){1}}
 \put(3,3){\vector(0,1){1}}
 \put(5,3){\vector(0,1){1}}
 \put(7,3){\vector(0,1){1}}
 \put(3,4.2){\makebox(0,0){$s$-$\sigma$-$m_{1}$}}
 \put(5.1,4.2){\makebox(0,0){$s$-$\sigma$-$m_{2}$}}
 \put(7.3,4.2){\makebox(0,0){$s$-$\sigma$-$m_{3}$}}
 \put(3.5,4.2){\vector(1,0){1}}
 \put(5.7,4.2){\vector(1,0){1}.}
\end{picture}

\vskip 1.4cm\setlength{\parindent}{0.5cm}

A space $X$ is called {\it a stratifiable or $M_{3}$-space} if it has a $\sigma$-cushioned pairbase. By \cite{MT}, we have that $X$ is $M_{1}$-space iff $X$ is $M_{3}$ and $m_{1}$ iff $X$ is $M_{3}$ and $s$-$m_{1}$. Moreover, if we let $X$ be a regular stratifiable space, then all the spaces on the above are equivalent, see \cite{BR, IM, MT}.

For each space $X$, we let $I(X)$ be the set of all isolated points of $X$. If $X$ is scattered, then let $X_{0}=X$; proceeding inductively assume that $\alpha$ is an ordinal and we constructed $X_{\beta}$ for all $\beta <\alpha$. If $\alpha =\beta +1$ for some $\beta$, then we let $X_{\alpha}=X_{\beta}\setminus I(X_{\beta})$. If $\alpha$ is a limit ordinal, then we let $X_{\alpha}=\bigcap\{X_{\beta}: \beta <\alpha\}$. The first ordinal $\alpha$ such that $X_{\alpha}=\emptyset$ is called the {\it dispersion index of $X$} and is denoted by $i(X)$, see \cite{DA}.

Reader may refer to
\cite{E1989, Gr1984} for notations and terminology not
explicitly given here.
\bigskip

\section{$s$-$m_{1}$ and $s$-$\sigma$-$m_{1}$ spaces}
In \cite{DA}, A. Dow, R. Mart$\acute{\imath}$nez and V.V. Tkachuk proved that each space with a finite number of non-isolated points is an $m_{1}$-space, and each compact scattered space $X$ and $i(X)\leq 3$ is an $m_{1}$-space. However, we shall see that there exists an $m_{1}$-space $X$ such that $X$ is non-$s$-$m_{1}$, see Example~\ref{e0}. But we have the follow Theorems~\ref{t0} and~\ref{t3}, which extend the results of the above.

\begin{lemma}\label{l1}
Let $\mathcal{P}$ be a topological property such that (a) if space $X$ has $\mathcal{P}$ then $X$
has $m_{i}$ and (b) if $X$ has $\mathcal{P}$ and $A$ is a closed subspace then $X/A$ has $\mathcal{P}$. Then
every space $X$ with property $\mathcal{P}$ has the $s$-variant property $s$-$m_{i}$.
\end{lemma}

\begin{proof}
Take any space $X$ with property $\mathcal{P}$, and take any closed subset $A$ of $X$. Then, by (b),
$X/A$ has property $\mathcal{P}$, and so is $m_{i}$ (by (a)). In particular, the point $A$ in $X/A$
is an $m_{i}$ point, and so the set $A$ has a `nice' outer base in $X$. From which it
follows that $X$ has $s$-$m_{i}$.
\end{proof}

\begin{theorem}\label{t0}
Each space with a finite number of non-isolated points is an $s$-$m_{1}$-space.
\end{theorem}

\begin{proof}
Let $X$ be a space with a finite number of non-isolated points, and let $A$ be a closed subspace of $X$. It is easy to see that $X/A$ has finite number of non-isolated points. By \cite[Proposition 2.9]{DA}, a space with a finite number of non-isolated points is $m_{i}$, and thus $X$ is an $s$-$m_{1}$-space by Lemma~\ref{l1}.
\end{proof}

\begin{theorem}\label{t3}
If $X$ is a compact scattered space and $i(X)\leq 3$, then $X$ is an $s$-$m_{1}$-space.
\end{theorem}

\begin{proof}
Let $X$ be a compact scattered space and $i(X)\leq 3$, and let $A$ be a closed subspace of $X$. It is easy to see that $X/A$ is also a compact scattered space and $i(X)\leq 3$. By \cite[Theorem 3.1]{DA}, a compact scattered space and $i(X)\leq 3$ is $m_{i}$, and thus $X$ is an $s$-$m_{1}$-space by Lemma~\ref{l1}.
\end{proof}

The proofs of the following Propositions~\ref{p0}, ~\ref{p1} and~\ref{p2} are easy, and so we omit them.

\begin{proposition}\label{p0}
If $X$ has a clopen closure-preserving neighborhood base at any closed set then $X$ is hereditarily $s$-$m_{1}$. In particular, any extremally disconnected $s$-$m_{1}$ space is hereditarily $s$-$m_{1}$.
\end{proposition}

\begin{proposition}\label{p1}
Suppose that $X$ is a space and a closed set $F\subset X$ has an open neighborhood base in $X$ which is well-ordered by the reverse inclusion. Then $X$ is $s$-$m_{1}$ at $F$.
\end{proposition}

\begin{proposition}\label{p2}
If $X$ is an $s$-$m_{1}$-space, and $D$ is dense in $X$. Then $D$ is also an $s$-$m_{1}$-subspace.
\end{proposition}

A map $f: X\rightarrow Y$ is called {\it quasi-open} if, for each non-empty open subset $U$ of $X$, the interior of $f(U)$ is non-empty. $f$ is called an {\it irreducible map} if, for each proper closed subset $F$ of $X$, we have $f(F)\neq Y$.

\begin{lemma}\label{l0}\cite{GG}
Let $f: X\rightarrow Y$ be a quasi-open closed map. If $\mathscr{B}$ is a closure-preserving open family of $X$, then $\varphi =\{\mbox{int}(f(B)): B\in\mathscr{B}\}$ is a closure-preserving open family of $Y$.
\end{lemma}

\begin{theorem}\label{t1}
Let $f: X\rightarrow Y$ be a quasi-open closed map. If $X$ is an $s$-$m_{1}$-space, then $Y$ is also an $s$-$m_{1}$ space.
\end{theorem}

\begin{proof}
Let $F$ be any closed set of $Y$ and $F\neq Y$. Then $f^{-1}(F)$ is closed in $X$. Since $X$ is $s$-$m_{1}$,
$f^{-1}(F)$ has a closure-preserving open neighborhood base $\mathscr{B}$ at $f^{-1}(F)$. Since $f$ is a quasi-open closed map, the family $\varphi =\{\mbox{int}(f(B)): B\in\mathscr{B}\}$ is a closure-preserving open family of $Y$ by Lemma~\ref{l0}. Moreover, because $f$ is a closed map, we have $F\subset \mbox{int}(f(B))$ for each $B\in\mathscr{B}$. It is easy to see that $\varphi$ is an open neighborhood base at $F$ in $Y$.
\end{proof}

\begin{corollary}\label{c0}
Closed and irreducible maps preserve $s$-$m_{1}$-spaces.
\end{corollary}

\begin{proof}
Since closed and irreducible maps are quasi-open maps\cite{GG}, closed and irreducible maps preserve $s$-$m_{1}$ property by Theorem~\ref{t1}.
\end{proof}

Next, we shall give an example to show that there exists an $m_{1}$-space $X$ which is non-$s$-$m_{1}$. Firstly, we prove the following Theorem~\ref{t2}.

Let $A$ be a subset of a space $X$. We call a family
$\mathcal{N}$ of open subsets of $X$ is an {\it outer base} of $A$ in
$X$ if for any $x\in A$ and open subset $U$ with $x\in U$ there is a
$V\in \mathcal{N}$ such that $x\in V\subset U$.

\begin{theorem}\label{t2}
If $X$ is Eberlein compact then $X$ is an $s$-$\sigma$-$m_{1}$-space.
\end{theorem}

\begin{proof}
Let $F$ be any closed subset of $X$. Since $X$ is Eberlein compact, it follows from \cite[Theorem 3.13]{DA} that $F$ has a $\sigma$-closure-preserving outer base $\mathscr{B}=\bigcup_{n\in \mathbb{N}}\mathscr{B}_{n}$ in $X$, where, for each $n\in \mathbb{N}$, $\mathscr{B}_{n}$ is closure-preserving and $\mathscr{B}_{n}\subset \mathscr{B}_{n+1}.$ For each $n\in \mathbb{N}$, let $$\mathscr{P}_{n}=\{\bigcup\mathscr{B}^{\prime}: \mathscr{B}^{\prime}\ \mbox{is a finite subfamily of}\ \mathscr{B}_{n}\ \mbox{and}\ F\subset\bigcup\mathscr{B}^{\prime}\}.$$It is easy to see that $\mathscr{P}=\bigcup_{n\in \mathbb{N}}\mathscr{P}_{n}$ is a $\sigma$-closure-preserving local base at the set $F$ in $X$, where, for each $n\in \mathbb{N}$, $\mathscr{P}_{n}$ is closure-preserving.
\end{proof}

Recalled that a closed map $f: X\rightarrow Y$ which is {\it perfect} if, for each $y\in Y$, $f^{-1}(y)$ is compact.

\begin{example}\label{e0}
There exists an $m_{1}$-space $X$ such that the following conditions are satisfied:
\begin{enumerate}
\item $X$ is an $s$-$\sigma$-$m_{1}$-space, and non-$s$-$m_{1}$-space;

\item The image of $X$ under some perfect and irreducible map is not an $m_{1}$-space.
\end{enumerate}
\end{example}

\begin{proof}
Let $X$ be the Alexandroff double $D$ of the Cantor set $C$. Then $X$ is first countable Elberlein compact space \cite{DA}. Hence $X$ is $s$-$\sigma$-$m_{1}$ by Theorem~\ref{t2}. Let $f: X\rightarrow Y$ be the quotient map by identifying the non-isolated point of $X$ to one point. Then $f$ is an irreducible and perfect map. However, $Y$ is not an $m_{1}$-space by \cite[Corollary 3.18]{DA}, and hence $X$ is a not an $s$-$m_{1}$-space by Corollary~\ref{c0}. Moreover, it is easy to see that first-countable spaces are $m_{1}$. However, $X$ is non-$s$-$m_{1}$-space. Therefore, compact first-countable is not need to be an $s$-$m_{1}$-space.
\end{proof}

In \cite{DA}, the authors prove that each $GO$ space is $m_{1}$. However, we don't know whether each $GO$ space is $s$-$m_{1}$, and so we have the following question.

\begin{question}
Let $X$ be a $GO$ space. Is $X$ $s$-$m_{1}$?
\end{question}

\bigskip

\section{$s$-$m_{2}$ and $s$-$\sigma$-$m_{2}$ spaces}
Since closed maps preserve a closure-preserving family, we have the following theorem.

\begin{theorem}
Closed maps preserve $s$-$m_{2}$ and $s$-$\sigma$-$m_{2}$ spaces, respectively.
\end{theorem}

\begin{theorem}
If $X$ is an $s$-$m_{2}$-space and $Y\subset X$ then $Y$ is $s$-$m_{2}$.
\end{theorem}

A space $X$ is {\it monotonically normal} if there is a function $G$ which assigns to each closed ordered pair $(H, K)$ of disjoint closed subsets of $X$ an open subset $G(H, K)\subset X$ such that:

(1) $H\subset G(H, K)\subset \overline{G(H, K)}\subset X\setminus K$;

(2) $G(H, K)\subset G(H^{\prime}, K^{\prime})$ for disjoint closed subsets $H^{\prime}$ and $K^{\prime}$ with $H\subset H^{\prime}$ and $K\supset K^{\prime}$;

(3) $G(H, K)\cap G(K, H)=\emptyset$;

Moreover, if $X$ also satisfies the following condition

(4) if $H^{\prime}\subset G(H, K)$ with $H^{\prime}$ closed in $X$, then $G(H^{\prime}, K)\subset G(H, K)$. \\ Then $X$ is called {\it strongly monotonically normal} \cite{HR, HR1}.

In \cite{DA}, the authors pose the following question.

\begin{question}\cite{DA}\label{q0}
Must every monotonically normal space be an $m_{1}$-space?
\end{question}

Next, we shall give some partial answer for this Question~\ref{q0}, see Theorem~\ref{t5}.

\begin{theorem}\label{t5}
Let $X$ be a strongly monotonically normal space. Then $X$ is an $s$-$m_{2}$-space.
\end{theorem}

\begin{proof}
Let $A$ be a closed subspace of $X$. In \cite{BR1}, R.E. Buck, R.W. Heath and P.L. Zenora showed that a closed image of a strongly monotonically normal space is again strongly monotonically normal, and hence $X/A$ is strongly monotonically normal. By \cite[Theorem 3.13]{BR}, a strongly monotonically normal space is $m_{2}$, and thus $X$ is an $s$-$m_{2}$-space by Lemma~\ref{l1}.
\end{proof}

\begin{question}
Let $X$ be a strongly monotonically normal space. Is $X$ an $s$-$m_{1}$-space or an $m_{1}$-space?
\end{question}

\bigskip

\section{$s$-$m_{3}$ and $\sigma$-$m_{3}$ spaces}
The proofs of the following two theorems are obvious, and so we omit them.

\begin{theorem}
Closed maps preserve $s$-$m_{3}$ and $s$-$\sigma$-$m_{3}$-spaces, respectively.
\end{theorem}

\begin{theorem}
If $X$ is an $s$-$m_{3}$-space and $Y\subset X$ then $Y$ is $s$-$m_{3}$.
\end{theorem}

The following theorem is also a partial answer for the Question~\ref{q0}.

\begin{theorem}\label{t4}
Let $X$ be a monotonically normal space. Then $X$ is an $s$-$m_{3}$-space.
\end{theorem}

\begin{proof}
Let $A$ be a closed subspace of $X$. It is well known that monotonically normal spaces are preserved by closed images, and implies $m_{3}$. Hence $X$ is $m_{3}$ and $X/A$ is monotonically normal, which follows that $X$ is an $s$-$m_{3}$-space by Lemma~\ref{l1}.
\end{proof}

Let $X$ be a space and $\kappa$ an infinite cardinal. For each $x\in X$, we denote $t(x, X)$ means that for any $A\subset X$ with $x\in \overline{A}$ there exists a set $B\subset A$ such that $|B|\leq \kappa$ and $x\in \overline{B}$; moreover, $t(X)\leq \kappa$ iff $t(x, X)\leq \kappa$ for each $x\in X$. The space $X$ with $t(X)\leq \kappa$ are said to have tightness $\leq\kappa$.

A pairwise disjoint collection of non-empty open subsets in $X$ is called a {\it cellular family}. The {\it cellularity of $X$}, defined as follows:$$c(X)=\mbox{sup}\{|\mathcal{U}|: \mathcal{U}\ \mbox{is a cellular family in}\ X\}+\omega.$$

In \cite{CP}, P.M. Gartside proved that for each monotonically normal space $X$, we have $t(X)\leq c(X)$. We shall extend this result of P.M. Gartside, and prove that, for each $\sigma$-$m_{3}$-space $X$, we have $t(X)\leq c(X)$.

\begin{theorem}
Suppose that a space $X$ is $\sigma$-$m_{3}$ at some point $x\in X$ and $\kappa$ is an infinite cardinal such that $c(U)\leq\kappa$ for some open neighborhood $U$ of the point $x$. Then $t(x, X)\leq\kappa$. In particular, if $X$ is an $\sigma$-$m_{3}$-space, then $t(X)\leq c(X)$.
\end{theorem}

\begin{proof}
Fix any set $A\subset X\setminus\{x\}$ with $x\in \overline{A}$.
Let $\mathscr{P}=\bigcup_{n\in \mathbb{N}}\mathscr{P}_{n}$ be a $\sigma$-cushioned pairbase at the point $x$,
where $\mathscr{P}_{n}\subset \mathscr{P}_{n+1}$ for each $n\in \mathbb{N}$, and for
each $P\in \mathscr{P}, P_{2}$ is closed in $X$. Without loss of generality, we may assume that $A\subset U$ and $\bigcup_{n\in \mathbb{N}}(\bigcup\mathscr{P}_{n})\subset U$. For
each $y\in X\setminus\{x\}$ and $n\in \mathbb{N}$, put $$W_{ny}=X\setminus\overline{\bigcup\{P_{1}: P\in\mathscr{P}\ \mbox{and}\ y\not\in P_{2}\}}.$$For each $n\in \mathbb{N}$, since $y\in X\setminus\bigcup\{P_{2}: P\in\mathscr{P}\ \mbox{and}\ y\not\in P_{2}\}\subset W_{ny}$, $W_{ny}$ is an open neighborhood of $y$.

Claim 1: If $Q\subset X\setminus\{x\}$, then $x\in \overline{Q}$ if and only if, for each $n\in \mathbb{N}$,
$x\in \overline{\bigcup\{W_{ny}: y\in Q\}}$.

In fact, if $x\in\overline{Q}$, then, for each open neighborhood $V$ of point $x$, we have $V\cap Q\neq\emptyset$. Choose a point $y\in V\cap Q$. It follows that $V\cap W_{ny}\neq\emptyset$ for each $n\in \mathbb{N}$, and hence we have $x\in \overline{\bigcup\{W_{ny}: y\in Q\}}$. For each $n\in \mathbb{N}$, let $x\in \overline{\bigcup\{W_{ny}: y\in Q\}}$. Suppose that $x\not\in \overline{Q}$. Then there exist an $n\in \mathbb{N}$ and $P\in \mathscr{P}_{n}$ such that $x\in P_{1}$ and $P_{2}\cap Q=\emptyset$. For each $y\in Q$, since $y\not\in P_{2}$, we have $W_{ny}\cap P_{1}=\emptyset$, and so $x\not\in\overline{\bigcup\{W_{ny}: y\in Q\}}$, which is a contradiction.

For each $n\in \mathbb{N}$ and $y\in A$, put $G_{ny}=W_{ny}\cap U$. It follows from $c(U)\leq\kappa$ that we can choose a set $D_{n}\subset A$ such that $|D_{n}|\leq\kappa$ and $G_{n}=\bigcup\{G_{ny}: y\in D_{n}\}$ is dense in $H_{n}=\bigcup\{G_{ny}: y\in A\}$. Let $D=\bigcup_{n\in \mathbb{N}}D_{n}$. Then $|D|\leq\kappa$. Obviously, we have that $$x\in\overline{A}\subset \bigcap_{n\in \mathbb{N}}\overline{H_{n}}=\bigcap_{n\in \mathbb{N}}\overline{G_{n}},$$which implies that $x\in\overline{G_{n}}\subset\overline{\bigcup\{W_{ny}: y\in D\}}$ for each $n\in \mathbb{N}$. By the Claim 1, we have $x\in \overline{D}$. Therefore, $t(x, X)\leq\kappa$.
\end{proof}

\begin{corollary}\cite{DA}
If $X$ is an $m_{2}$-space, then $t(X)\leq c(X)$.
\end{corollary}

\begin{corollary}\cite[Theorem 10]{CP}
If $X$ is a monotonically normal space, then $t(X)\leq c(X)$.
\end{corollary}

Recall that a family $\mathcal{U}$ of non-empty open sets of a space
$X$ is called a {\it $\pi$-base} if for each non-empty open set $V$
of $X$, there exists an $U\in\mathcal{U}$ such that $V\subset U$. The
{\it $\pi$-character} of $x$ in $X$ is defined by $\pi_{\chi}(x,
X)=\mbox{min}\{|\mathcal{U}|:\mathcal{U}\ \mbox{is a local}\
\pi\mbox{-base at}\ x\ \mbox{in}\ X\}$. The {\it $\pi$-character of
$X$} is defined by $\pi_{\chi}(X)=\mbox{sup}\{\pi_{\chi}(x, X):x\in
X\}$.

In \cite{DA}, A. Dow, R. Ram$\acute{\imath}$rez and V.V. Tkachuk proved that if $X$ is a separable $m_{2}$-space with the Baire property then $X$ has countable $\pi$-character. However, we find the proof has a gap. Next, we shall give out the correct proof. In fact, we have more general result, see Theorem~\ref{t6}.

\begin{theorem}\label{t6}
Suppose that a space $X$ is a separable space with the Baire property. If $X$ is $\sigma$-$m_{3}$ at some point $x\in X$, then it has countable $\pi$-character at the point $x$.
\end{theorem}

\begin{proof}
Suppose that $x$ is a non-isolated point in $X$, and that $D=\{d_{n}: n\in \mathbb{N}\}\subset X\setminus\{x\}$ is a dense subset of $X$. Let $\mathscr{P}=\bigcup\mathscr{P}_{n}$ be a $\sigma$-cushioned pairbase at the point $x$, where, for each $n\in \mathbb{N}$, $\mathscr{P}_{n}$ is cushioned. For each $n\in \mathbb{N}$, put $D_{n}=D\cap \bigcup\{P_{1}: P\in\mathscr{P}_{n}\}=\{d_{n, m}: m\in \mathbb{N}\}$. For each $y\in X\setminus\{x\}$ and $n\in \mathbb{N}$, put $$W_{ny}=X\setminus\overline{\bigcup\{P_{1}: P\in \mathscr{P}_{n}\ \mbox{and}\ y\not\in P_{2}\}}.$$ Then $W_{ny}$ is an open neighborhood of $y$. For each $n, m\in \mathbb{N}$, put $B_{n, m}=\{y\in X\setminus\{x\}: d_{n, m}\in W_{ny}\}$. It is easy to see that $B_{n, m}\subset \bigcap\{P_{2}: d_{n, m}\in P_{1}, P\in\mathscr{P}_{n}\}.$ Further, for each $n, m\in \mathbb{N}$, $B_{n, m}$ is closed in $X\setminus\{x\}$. In fact, let $y\in \overline{B_{n, m}}\setminus\{x\}$. Then $y\in W_{ny}$, and hence $W_{ny}\cap B_{n, m}\neq\emptyset$. For each $z\in W_{ny}\cap B_{n, m}$, we have $W_{nz}\subset W_{ny}$, and thus $d_{n, m}\in W_{nz}\subset W_{ny}$. So $y\in B_{n, m}$.
Let $\mathscr{U}=\{\mbox{int}(B_{n, m}): n, m\in \mathbb{N}\}.$ Then $\mathscr{U}$ is a countable $\pi$-base at the point $x$.

In fact, let any open subset $U$ with $x\in U$. Since $\mathscr{P}$ is a local pairbase at the point $x$, there exist an $n\in \mathbb{N}$ and $P\in \mathscr{P}_{n}$ such that $x\in P_{1}\subset P_{2}\subset U$, where $P=(P_{1}, P_{2})$. It is easy to see that, for each $m\in \mathbb{N}$, $d_{n, m}\in B_{n, m}\subset P_{2}$ whenever $d_{n, m}\in P_{1}.$
Let $V=P_{1}\setminus\{x\}$. Then $V$ is a non-empty open subset and has the Baire property. If $y\in V$, then $D_{n}\bigcap W_{ny}\bigcap V\neq\emptyset$, and hence there exists an $m\in \mathbb{N}$ such that $d_{n, m}\in W_{ny}\cap V$, i.e., $y\in B_{n, m}$. Therefore, $V\subset \bigcup\{B_{n, m}: d_{n, m}\in P_{1}\}$. Since $V$ has the Baire property, there exists an $m\in \mathbb{N}$ such that $G=\mbox{int}(B_{n, m})\neq\emptyset$. Hence $G\in\mathscr{U}$ and $G\subset B_{n, m}\subset P_{2}\subset U$. Therefore, $\mathscr{U}$ is a countable $\pi$-base at the point $x$.
\end{proof}

\begin{corollary}\cite{DA}
Suppose that a space $X$ is a separable space with the Baire property and $X$ is $m_{2}$ at some point $x\in X$. Then $X$ has countable $\pi$-character at the point $x$.
\end{corollary}

\begin{corollary}
If $G$ is a separable $\sigma$-$m_{3}$-group with the Baire property then $G$ is metrizable.
\end{corollary}

\begin{proof}
Since $w(G)=\pi w(G)$ for any topological group, $G$ is metrizable.
\end{proof}

{\bf Acknowledgements}. We wish to thank
the reviewers for the detailed list of corrections, suggestions to the paper, and all her/his efforts
in order to improve the paper. In particular, Lemma~\ref{l1} is due to the reviewers, which gives an easy proofs for Theorems~\ref{t0}, \ref{t3}, \ref{t5} and~\ref{t4} in our original paper.

\end{document}